\documentclass[11pt]{amsart}
\usepackage {enumerate}
\usepackage{setspace}
\usepackage{caption}
\usepackage{mathrsfs}
\usepackage{hyperref}
\usepackage{esint}
\usepackage{amssymb}

\usepackage{color}
\usepackage{lipsum}
\newcommand\blfootnote[1]{%
  \begingroup
  \renewcommand\thefootnote{}\footnote{#1}%
  \addtocounter{footnote}{-1}%
  \endgroup
}
\newtheorem{theorem}{Theorem}[section]
\newtheorem{lemma}[theorem]{Lemma}
\newtheorem{corollary}[theorem]{Corollary}
\newtheorem{proposition}[theorem]{Proposition}

\newtheorem{conjecture}[theorem]{Conjecture}
\numberwithin{equation}{section}

\theoremstyle {definition}
\newtheorem{definition}[theorem]{Definition}
\newtheorem{remark}[theorem]{Remark}

\DeclareMathOperator{\area}{area}
\DeclareMathOperator{\Hess}{Hess}

\DeclareMathOperator{\Ric}{Ric}

\DeclareMathOperator{\Div}{div}
\DeclareMathOperator{\Sec}{sec}
\DeclareMathOperator{\width}{width}
\DeclareMathOperator{\lip}{Lip}
\DeclareMathOperator{\dist}{dist}
\DeclareMathOperator{\pt}{pt}

\begin{document}
\title{Width estimate and doubly warped product}
\author{Jintian Zhu}
\address{Key Laboratory of Pure and Applied Mathematics, School of Mathematical Sciences, Peking University, Beijing, 100871, P.~R.~China}
\email{zhujt@pku.edu.cn, jintian@uchicago.edu}
\maketitle
\begin{abstract}
In this paper, we give an affirmative answer to Gromov's conjecture (\cite[Conjecture E]{Gromov2018}) by establishing an optimal Lipschitz lower bound for a class of smooth functions on connected orientable open $3$-manifolds with uniformly positive sectional curvatures. For rigidity we show that if the optimal bound is attained the given manifold must be a quotient space of $\mathbf R^2\times(-c,c)$ with some doubly warped product metric. This gives a characterization for doubly warped product metrics with positive constant curvature. As a corollary, we also obtain a focal radius estimate for immersed toruses in $3$-spheres with positive sectional curvatures.
\end{abstract}
\section{Introduction}
Let $M^n$ be a connected orientable compact manifold with non-empty boundary $\partial M$, which is the disjoint union of connected components $\partial_-M$ and $\partial_+M$. For any smooth metric $g$ on $M$, the {\it width} of $(M,g)$ is defined as
\begin{equation}
\width(M,g)=\dist_g(\partial_+M,\partial_-M).
\end{equation}

In his paper \cite{Gromov2018}, Gromov introduced the following definition.
\begin{definition}\label{Defn: overtorical band}
Let $
M_0=T^{n-1}\times [-1,1]$
 and
$ \partial_\pm M_0=T^{n-1}\times \{\pm 1\}$. If $M$ admits a continuous map
$
f:(M,\partial_\pm M)\to (M_0,\partial_\pm M_0)
$
with nonzero degree, $M$ is called an overtorical band.
\end{definition}
With this notion, he proved the following width estimate.
\begin{theorem}[\cite{Gromov2018}]\label{Thm: width scalar}
For $2\leq n\leq 8$, let $(M^n,g)$ be a smooth overtorical band with its scalar curvature $R(g)\geq n(n-1)\sigma^2$ for some $\sigma>0$. Then
\begin{equation}\label{Eq: width scalar}
\width(M,g)\leq \frac{2\pi}{n\sigma}.
\end{equation}
\end{theorem}
This estimate is optimal and it is believed that any band with the extreme width is isometric to an open torical band with warped product metric. That is, up to scaling the band must be
$$
M= \left(-\frac{\pi}{n},\frac{\pi}{n}\right)\times T^{n-1}\quad\text{\rm and}\quad g=\mathrm dt^2+\cos^{\frac{4}{n}}\left(\frac{n}{2}t\right)g_{flat},
$$
where $g_{flat}$ is an arbitrary flat metric on $T^{n-1}$.

For bands with positive sectional curvatures, Gromov made the following conjecture:

\begin{conjecture}[\cite{Gromov2018}, Conjecture E]\label{Conj: main}
Let $g$ be a smooth metric on $M=T^2\times [-1,1]$ with $\sec(g)\geq 1$. Then
$$
\width(M,g)\leq \frac{\pi}{2}.
$$
\end{conjecture}
The main purpose of this paper is to give an affirmative answer to this conjecture by establishing an optimal Lipschitz lower bound for a class of smooth functions on connected orientable open $3$-manifolds as in \cite{Gromov2019}. Apart from the inequality, we also prove a rigidity result for our estimate. It turns out that up to rescaling the equality forces the given manifold to be $\mathring M_0/\Gamma$, where $\mathring M_0=(-\frac{\pi}{4},\frac{\pi}{4})\times\mathbf R^2$ is equipped with the doubly warped product metric
$$
g=\mathrm dt^2+\sin^2\left(t+\frac{\pi}{4}\right)\mathrm ds_1^2+\cos^2\left(t+\frac{\pi}{4}\right)\mathrm ds_2^2
$$
and $\Gamma$ is a lattice of $\mathbf R^2$. Obviously, the $t$-component in above expression induces a distance function on $\mathring M_0/\Gamma$. For convenience, we call it the {\it standard signed distance function} on $\mathring M_0/\Gamma$.

For a connected orientable open manifold $\mathring M^3$, we define
\begin{equation}
\mathcal F_{\mathring M}=\left\{\phi\in C^\infty(\mathring M,(-1,1))\left|\begin{array}{c}\text{ $\phi$ is surjective and proper }\\\text{ and $\phi^*([\pt])$ is non-spherical}
\end{array}\right.\right\},
\end{equation}
where $\phi^*([\pt])$ denotes the pull-back of the homology class $[\pt]$ and non-spherical means it cannot be represented by a $2$-cycle with only spherical components.

Our main theorem is
\begin{theorem}\label{Thm: main}
Assume $(\mathring M^3,g)$ is a connected orientable open Riemannian manifold with uniformly positive sectional curvatures and non-empty $\mathcal F_{\mathring M}$. Then we have
\begin{equation}\label{Eq: main}
\left(\inf \Sec(g)\right)^{\frac{1}{2}}\cdot\lip\phi\geq \frac{4}{\pi}
\end{equation}
for any $\phi\in \mathcal F_{\mathring M}$. Up to rescaling the equality holds if and only if $(\mathring M,g)$ is isometric to $\mathring M_0/\Gamma$ for a lattice $\Gamma$ of $\mathbf R^2$ and $\phi$ is a multiple of the standard signed distance function on $\mathring M_0/\Gamma$.
\end{theorem}
\begin{remark}
It was pointed out to the author by Chao Li and the referee that the main theorem still holds when $(\mathring M,g)$ has uniformly positive Ricci curvatures. We will include all necessary modifications in the last section. As a result, this also leads to certain improvements for the corollaries below. For simplicity we just omit further discussions since they are almost direct.
\end{remark}

Several corollaries can be derived from our main theorem. The first one is the desired width estimate for overtorical bands with positive sectional curvatures.
\begin{corollary}\label{Cor: width estimate}
Let $(M^3,g)$ be a smooth overtorical band with positive sectional curvatures. Then
\begin{equation}\label{Eq: width sectional}
\left(\inf \Sec(g)\right)^{\frac{1}{2}}\cdot\width(M,g)\leq \frac{\pi}{2}.
\end{equation}
\end{corollary}

The second one is a focal radius estimate for immersed toruses in $3$-spheres with positive sectional curvatures. Recall that the focal radius of any immersed surface $i:\Sigma\to M$ is defined as
\begin{equation*}
r_f(\Sigma,M)=\sup\{r>0:\text{map $\exp^\bot:\nu_{\Sigma,r}\to M$ has no critical point}\},
\end{equation*}
where
$\nu_\Sigma$ is the total space of the pull-back normal bundle on $\Sigma$ and
$$
\nu_{\Sigma,r}=\{(x,v)\in\nu_\Sigma:|v|<r\}.
$$
We have
\begin{corollary}\label{Cor: focal estimate}
Let $(M^3,g)$ be a $3$-sphere with $\sec(g)\geq 1$ and $i:\Sigma\to M$ an immersed torus. Then the focal radius of $\Sigma$ satisfies
\begin{equation}\label{Eq: focal estimate}
r_f(\Sigma,M)\leq \frac{\pi}{4}.
\end{equation}
The equality holds if and only if $g$ has constant curvature $1$ and $i(\Sigma)$ is congruent to the Clifford torus.
\end{corollary}

Generalizations of these results in higher dimensions are interesting and we make the following conjectures:
\begin{conjecture}\label{Conj: focal radius high dimension}
Let $(M^n,g)$ be a $n$-sphere with $\sec(g)\geq 1$, where $n=p+q+1$ for $p,q\in\mathbf N_+$. Suppose that $i:\Sigma\approx\mathbf S^p\times \mathbf S^q\to M$ is an immersed hypersurface, then the focal radius of $\Sigma$ satisfies
$$
r_f(\Sigma,M)\leq \frac{\pi}{4},
$$
where the equality holds if and only if $g$ has constant sectional curvature $1$ and $i(\Sigma)$ is congruent to the Clifford hypersurface $\mathbf S^p(\frac{1}{\sqrt 2})\times \mathbf S^q(\frac{1}{\sqrt 2})$.
\end{conjecture}
\begin{remark}
If either $p$ or $q$ is equal to $1$, this conjecture is true from the discussion in the end of Section \ref{Sec: proof corollaries}, where we present a proof from professor Andr\'e Neves using Jacobi equation for geodesics and Cheeger-Gromoll splitting theorem.
\end{remark}

With the name {\it Clifford band} for manifolds $M_{p,q}=\mathbf S^p\times \mathbf S^q\times [-1,1]$, we also speculate
\begin{conjecture}
If $(M_{p,q},g)$ be a smooth Clifford band with its sectional curvature $\Sec(g)\geq 1$, then
\begin{equation*}
\width(M,g)\leq \frac{\pi}{2}.
\end{equation*}
\end{conjecture}

Now let us say some words on our proof for Theorem \ref{Thm: main}. Basically the proof follows the line of $\mu$-bubble method in \cite{Gromov2019} and the inequality \eqref{Eq: main} comes from a classical analysis on the second variation formula. However, the rigidity is very subtle due to the lack of compactness on open manifolds for minimizing $\mu$-bubbles. In this case, we follow a similar idea from \cite{CEM2019} to choose our functional carefully such that the approximating minimizing $\mu$-bubbles always intersect with a fixed compact subset, which is the key to make a limiting procedure possible. Once this has been done, the rigidity result comes from a standard foliation argument from \cite{BBN2010} (see also \cite{Z2019}). We emphasize that our idea can be also applied to establish a rigidity result for Theorem \ref{Thm: width scalar}.

The rest part of this paper will be organized as follows. In section 2, we show the inequality \eqref{Eq: main}. In section 3, we handle the rigidity case for our main theorem. In section 4, we give a proof for Corollary \ref{Cor: width estimate} and Corollary \ref{Cor: focal estimate}. In section 5, we present necessary modifications for our main theorem under positive Ricci curvature condition.

\section*{Acknowledgement}
This work is partially supported by China Scholarship Council and the NSCF grants No. 11671015 and 11731001. The author would like to thank Professor Andr\'e Neves for many helpful conversations. He is also grateful to Professor Yuguang Shi for constant encouragements. He thanks Chao Li and the referee for valuable comments and suggestions as well.

\section{Proof for \eqref{Eq: main}}\label{Sec: proof 1.4}
In this section, $(\mathring M,g)$ denotes a connected orientable open $3$-manifold with nonempty $\mathcal F_{\mathring M}$. Let $\phi$ be a fixed element in $\mathcal F_{\mathring M}$. For any smooth function $h:(-T,T)\to \mathbf R$ with $0<T<1$, we define the following functional
\begin{equation}
\mathcal A^h(\Omega)=\mathcal H^2(\partial^*\Omega)-\int_{\mathring M}(\chi_\Omega-\chi_{\Omega_0})h\circ\phi\,\mathrm d\mathcal H^3,\quad \Omega_0=\{\phi<0\},
\end{equation}
where $\Omega$ is any Caccippoli set of $\mathring M$ with reduced boundary $\partial^*\Omega$ such that
$$\Omega\Delta\Omega_0\Subset \mathcal D(h\circ\phi)= \{-T<\phi<T\}$$
and $\chi_\Omega$ is the characteristic function of region $\Omega$. To avoid possible confusion on mean curvature we make a convention that the unit sphere $\mathbf S^2$ in $\mathbf R^3$ has mean curvature $2$ with respect to the outer unit normal.

First we prove the following proposition which is fundamental to our proof.

\begin{proposition}\label{Prop: existence of mu bubble}
Assume that $\pm T$ are regular values of $\phi$ and that the function $h$ satisfies
\begin{equation}\label{Eq: h blow up}
\lim_{t\to -T}h(t)=+\infty\quad \text{\rm and}\quad\lim_{t\to T}h(t)=-\infty,
\end{equation}
then
there exists a smooth minimizer $\hat\Omega$ for $\mathcal A^h$ such that $\hat\Omega\Delta\Omega_0\Subset \mathcal D(h\circ\phi)$.
\end{proposition}
\begin{proof}
Define
\begin{equation*}
\mathcal C(\mathring M)=\{\text{Caccippoli sets $\Omega$ in $\mathring M$ such that $\Omega\Delta \Omega_0\Subset \mathcal D(h\circ\phi)$}\}
\end{equation*}
and
\begin{equation*}
I=\inf\{\mathcal A^h(\Omega):{\Omega\in \mathcal C(\mathring M)}\}.
\end{equation*}
First we show $I>-\infty$. For any $s>0$, denote
$$
\Sigma^\pm_s=\{x\in \mathcal D(h\circ\phi):\dist(x,\phi^{-1}(\pm T))=s\}.
$$
Since $\pm T$ are regular values of $\phi$, $\Sigma^\pm_s$ becomes a foliation around $\phi^{-1}(\pm T)$ when $s$ is small. From \eqref{Eq: h blow up} we can assume $H^-_s\leq h\circ\phi$ and $H^+_s\leq-h\circ\phi$ for $s\leq s_0$, where $s_0$ is a small positive constant and $H^\pm_s$ is the mean curvature of $\Sigma^{\pm}_s$ with respect to $\partial_s$. Let $\Omega^\pm_s$ be the region enclosed by $\Sigma^\pm_s$ and $\phi^{-1}(\pm T)$. Possibly decreasing the value of $s_0$, we can construct a smooth vector field $X$ such that $X=\partial_s$ on $\Omega^\pm_{s_0}$. It is clear that
$$
\Div_g X=H_s^-\leq h\circ\phi\quad \text{in}\quad \Omega_{s_0}^-
$$
and
$$
\Div_g X=H_s^+\leq -h\circ\phi\quad \text{in}\quad \Omega_{s_0}^+.
$$
Notice that for any region $\Omega\in \mathcal C(\mathring M)$ we have the following estimate
\begin{equation*}
\begin{split}
&\mathcal A^h(\Omega\cup\Omega^-_{s_0}\backslash\Omega^+_{s_0})-\mathcal A^h(\Omega)\\=&\mathcal H^{2}(\partial\Omega^-_{s_0}\backslash \Omega)-\mathcal H^{2}(\partial^*\Omega\cap\Omega^-_{s_0})+\mathcal H^{2}(\partial\Omega^+_{s_0}\cap \Omega)\\
&\qquad -\mathcal H^{2}(\partial^*\Omega\cap\Omega^+_{s_0})-\int_{\Omega^-_{s_0} \backslash\Omega}h\circ\phi\,\mathrm d\mathcal H^3_g+\int_{\Omega\cap\Omega^+_{s_0}}h\circ\phi\,\mathrm d\mathcal H^3_g\\
\leq&\mathcal H^{2}(\partial\Omega^-_{s_0}\backslash \Omega)-\mathcal H^{2}(\partial^*\Omega\cap\Omega^-_{s_0})+\mathcal H^{2}(\partial\Omega^+_{s_0}\cap \Omega)\\
&\qquad -\mathcal H^{2}(\partial^*\Omega\cap\Omega^+_{s_0})-\int_{\Omega^-_{s_0} \backslash\Omega}\Div_gX\,\mathrm d\mathcal H^3_g-\int_{\Omega\cap\Omega^+_{s_0}}\Div_gX\,\mathrm d\mathcal H^3_g\\
\leq &0,
\end{split}
\end{equation*}
since
\begin{equation*}
\begin{split}
\int_{\Omega^-_{s_0} \backslash\Omega}\Div_gX\,\mathrm d\mathcal H^3_g&=\int_{\partial^*(\Omega^-_{s_0} \backslash\Omega)}\langle X,\nu\rangle_g\,\mathrm d\mathcal H^{2}_g\\
&\geq \mathcal H^{2}(\partial\Omega^-_{s_0}\backslash \Omega)-\mathcal H^{2}(\partial^*\Omega\cap\Omega^-_{s_0})
\end{split}
\end{equation*}
and
\begin{equation*}
\begin{split}
\int_{\Omega\cap\Omega^+_{s_0}}\Div_gX\,\mathrm d\mathcal H^3_g&=\int_{\partial^*(\Omega\cap\Omega^+_{s_0})}\langle X,\nu\rangle_g\,\mathrm d\mathcal H^{2}_g\\
&\geq\mathcal H^{2}(\partial\Omega^+_{s_0}\cap \Omega) -\mathcal H^{2}(\partial^*\Omega\cap\Omega^+_{s_0}).
\end{split}
\end{equation*}
It follows
$$
\mathcal A^h(\Omega)\geq \mathcal A^h(\Omega\cup\Omega^-_{s_0}\backslash \Omega^+_{s_0})\geq -C\mathcal H^3(\mathcal D(h\circ\phi)),\quad\forall\,\Omega\in \mathcal C(\mathring M),
$$
where $C$ is a universal constant such that $|h\circ\phi|\leq C$ on $\mathcal D(h\circ\phi)-\Omega^-_{s_0}\cup\Omega^+_{s_0}$. Hence $I>-\infty$.

Now we establish the existence of a smooth minimizer for $\mathcal A^h$ in $\mathcal C(\mathring M)$. Let $\Omega_k$ be a sequence of regions in $\mathcal C(\mathring M)$ such that $\mathcal A^h(\Omega_k)\to I$ as $k\to\infty$. According to the discussion above we can assume $\Omega_k\Delta\Omega_0\subset \mathcal D(h\circ\phi)-\Omega^-_{s_0}\cup\Omega^+_{s_0}$. For $k$ large enough there holds
$$
\mathcal H^2(\partial^*\Omega_k)\leq I+1+C\mathcal H^3(\mathcal D(h\circ\phi)).
$$
From the compactness of Caccippoli sets, after taking the limit of $\Omega_k$ we can obtain $\hat\Omega\in \mathcal C(\mathring M)$ with $\mathcal A^h(\hat\Omega)=I$. The smoothness of $\partial\hat\Omega$ comes from the regularity theorem \cite[Theorem 2.2]{ZZ2018}.
\end{proof}

Now we give the proof for \eqref{Eq: main}.
\begin{proof}
Without loss of generality, we can assume $\inf\sec(g)=1$ from rescaling. If \eqref{Eq: main} is not true, then there is a $\phi\in \mathcal F_{\mathring M}$ with $\lip\phi <4/\pi$. With $\beta<1$ a constant to be determined later, we define
$$
h=-2\tan\left(\frac{\pi}{2\beta} t\right),\quad t\in(-\beta,\beta).
$$
If $\pm\beta$ are regular values of $\phi$, then we can apply Proposition \ref{Prop: existence of mu bubble} to obtain a smooth minimizer $\hat\Omega$ for functional $\mathcal A^h$. Since the boundary $\partial\hat\Omega$ is homologic to $\partial\Omega_0$ i.e. $\phi^{-1}(0)$, the non-spherical requirement on $\phi^*([\pt])$ implies that $\partial\hat\Omega$ has a component $\hat\Sigma$ with nonzero genus.
For any smooth function $\psi$ on $\hat\Sigma$, we take a smooth vector field $X$ such that $X=\psi\nu$ on $\hat\Sigma$ and $X$ vanishes outside a small neighborhood of $\hat\Sigma$, where $\nu$ is the outward unit normal vector field on $\hat\Sigma$. Denote $\Phi_t$ to be the flow generated by $X$, then we have
\begin{equation*}
\delta{\mathcal A}^h(\psi)=\left.\frac{\mathrm d}{\mathrm dt}\right|_{t=0} {\mathcal A}^h(\Phi_t(\hat\Omega))=\int_{\hat\Sigma}(\hat H-h\circ\phi)\psi\,\mathrm d\sigma_g=0,
\end{equation*}
where $\hat H$ is the mean curvature of $\hat\Sigma$ with respect to $\nu$. Since $\psi$ is arbitrary, we see $\hat H=h\circ\phi$ on $\hat\Sigma$.
The second variation formula yields
\begin{equation}\label{Eq: second variation formula}
\begin{split}
\delta^2{\mathcal A}^h(\psi,\psi)&=\left.\frac{\mathrm d^2}{\mathrm dt^2}\right|_{t=0}{\mathcal A}^h(\Phi_t(\hat\Omega))\\
&=\int_{\hat\Sigma}|\nabla\psi|^2-(\Ric(\nu,\nu)+|A|^2+\nu(h \circ \phi))\psi^2\,\mathrm d\sigma_g\\
&\geq 0.
\end{split}
\end{equation}

Now we deduce a contradiction from \eqref{Eq: second variation formula}. Since the sectional curvature $\sec(g)\geq 1$, we have $\Ric(\nu,\nu)\geq 2$ and
$|A|^2\geq (h\circ\phi)^2-\hat R+2$,
where $\hat R$ is the scalar curvature of $\hat\Sigma$ with the induced metric.
Notice also
$$
(h\circ\phi)^2+\nu(h\circ\phi)\geq 4\tan^2\left(\frac{\pi}{2\beta}t\right)-\pi\beta^{-1}\lip\phi\cos^{-2} \left(\frac{\pi}{2\beta}t\right).
$$
From Sard's theorem we can choose $\beta$ such that $\pm\beta$ are regular values of $\phi$ and $\pi\beta^{-1}\lip\phi <4$. Taking the testing function $\psi$ to be identically one, we conclude from \eqref{Eq: second variation formula} that
\begin{equation}
4\pi\chi(\hat\Sigma)=\int_{\hat\Sigma}\hat R\,\mathrm d\sigma_g\geq \int_{\hat\Sigma}4+(h\circ\phi)^2+\nu(h\circ\phi)\,\mathrm d\sigma_g>0,
\end{equation}
which leads to a contradiction.
\end{proof}

\section{The rigidity case}\label{Sec: rigidity}
In this section, $(\mathring M,g)$ is assumed to be a connected orientable open $3$-manifold with $\inf\sec(g)=1$ such that there is a $\phi\in\mathcal F_{\mathring M}$ satisfying $\lip\phi=4/\pi$. We are going to show that $(\mathring M,g)$ is isometric to the open manifold $\mathring M_0/\Gamma$ for some lattice $\Gamma$ and $\phi=4\rho/\pi$, where $\rho$ is the standard signed distance function on $\mathring M_0/\Gamma$.

The first lemma provides appropriate functions for us to construct approximating minimizing $\mu$-bubbles.
\begin{lemma}\label{Lem: function h epsilon}
For $\epsilon\in(0,1)$, there is a family of odd smooth functions $h_\epsilon$ defined on $(-T_\epsilon,T_\epsilon)$ such that
\begin{itemize}
\item $T_\epsilon\uparrow 1$ as $\epsilon\to 0$;
\item the derivative $h_\epsilon'$ is negative everywhere;
\item if $|t|<2/3$, then $4h_\epsilon'/\pi+h_\epsilon^2<-4$; if $2/3<|t|<1$, then $4h_\epsilon'/\pi+h_\epsilon^2>-4$.
\end{itemize}
\end{lemma}
\begin{proof}
The proof comes from an explicit construction. Define
$$\alpha_\epsilon(t)=t+\epsilon \left(\sin\left(\frac{\pi}{2}t\right)-\frac{\pi}{4}t\right)$$
and
$$h_\epsilon(t)=-2\tan\left(\frac{\pi}{2}\alpha_\epsilon(t)\right).$$
Clearly, $h_\epsilon$ is an odd smooth function defined on $(-T_\epsilon,T_\epsilon)$, where $T_\epsilon$ is the unique positive solution to the equation $\alpha_\epsilon(t)=1$. The behavior of $T_\epsilon$ comes from implicit function theorem applied to this equation. Through direct calculation, we have
$$
h_\epsilon'=-\pi\cos^{-2}(2\alpha_\epsilon)\left(1+\frac{\pi\epsilon}{4} \left(2\cos\left(\frac{\pi}{2}t\right)-1\right)\right)
$$
and
\begin{equation}\label{Eq: h epsilon}
\frac{4}{\pi}h_\epsilon'+h_\epsilon^2=-4-{\pi\epsilon}{\cos^{-2}(2\alpha_\epsilon)} \left(2\cos\left(\frac{\pi}{2}t\right)-1\right).
\end{equation}
The conclusion now follows from the fact $0<\epsilon<1$.
\end{proof}

From Sard's theorem, there is a sequence $\epsilon_k\to 0$ such that $\pm T_{\epsilon_k}$ are regular values of $\phi$. Let $h_k=h_{\epsilon_k}$. Using Proposition \ref{Prop: existence of mu bubble} we can find a minimizer $\hat\Omega_k$ for functional $\mathcal A^{h_{k}}$. Possibly passing to a subsequence, the sequence $\hat\Omega_k$ converges to a local minimizer $\hat\Omega$ for functional $\mathcal A^h$ with $h=-2\tan(\pi t/2)$ such that $\chi_{\hat\Omega_k}\to\chi_{\hat\Omega}$ in the $L^1_{loc}$ sense. Here local minimizer means that for any Caccippoli set $\Omega$ such that $\Omega\Delta\hat\Omega\Subset \mathring M$ we have $\mathcal A_U^h(\Omega)\geq \mathcal A_U^h(\hat\Omega)$ for any region $U$ with compact closure and $\Omega\Delta\hat\Omega\Subset U$, where
\begin{equation*}
\mathcal A^h_U(\Omega)=\mathcal H^2(\partial^*\Omega\cap U)-\int_{\Omega\cap U}h\circ\phi\,\mathrm d\mathcal H^3.
\end{equation*}

The following lemma is well-known, but we give a proof for completeness.
\begin{lemma}
Denote $\hat\mu_k$ and $\hat\mu$ to be the Radon measures associated with $\partial\hat\Omega_k$ and $\partial\hat\Omega$ respectively. Then $\hat\mu_k$ converges to $\hat\mu$ after passing to a subsequence. That is, for any continuous function $f$ with compact support there holds
\begin{equation}\label{Eq: Radon measure convergence}
\lim_{k\to\infty}\int_{\mathring M}f\,\mathrm d\hat\mu_k=\int_{\mathring M}f\,\mathrm d\hat\mu.
\end{equation}
\end{lemma}
\begin{proof}
Let $U$ be any region with compact closure in $\mathring M$. For any compact subset $K\subset U$, we denote
$$
K_\delta=\{x\in U:\dist(x,K)<\delta\},\quad \delta>0.
$$
There is a positive constant $\delta_0$ such that $K_\delta\subset U$ for any $\delta<\delta_0$. Since $\mathcal H^2(\partial\hat\Omega\cap \bar U)$ are bounded and $\mathcal H^3\left((\hat\Omega_k\Delta\hat\Omega)\cap\bar U\right)$ converges to $0$, for almost every $\delta$ we see
\begin{equation}\label{Eq: slicing measure zero}
\mathcal H^2(\partial\hat\Omega\cap\partial K_\delta)=0,
\end{equation}
and
\begin{equation}\label{Eq: slicing converge zero}
\mathcal H^2\left((\hat\Omega_k\Delta\hat\Omega)\cap \partial K_\delta\right)\to 0,\quad k\to \infty.
\end{equation}
Let $\Omega_k=(\hat\Omega\cap K_\delta)\cup (\hat\Omega_k-K_\delta)$ be a competitor for $\hat\Omega_k$. From $\mathcal A_U^{h_k}(\Omega_k)\geq \mathcal A_U^{h_k}(\hat\Omega_k)$ we obtain
$$
\mathcal H^2(\partial^*\Omega_k\cap U)\geq \mathcal H^2(\partial\hat\Omega_k\cap U)+\int _U(\chi_{\Omega_k}-\chi_{\hat\Omega_k})h_k\circ\phi\,\mathrm d\mathcal H^3_g.
$$
Notice that
\begin{equation*}
\begin{split}
\mathcal H^2(\partial^*\Omega_k\cap U)=\mathcal H^2(\partial\hat\Omega_k-&K_\delta)+\mathcal H^2(\partial\hat\Omega\cap K_\delta)\\&+\mathcal H^2(\partial\hat\Omega\cap\partial K_\delta)+\mathcal H^2\left((\hat\Omega_k\Delta\hat\Omega)\cap \partial K_\delta\right).
\end{split}
\end{equation*}
It follows from \eqref{Eq: slicing measure zero}, \eqref{Eq: slicing converge zero} and $L^1$ convergence that
$$
\mathcal H^2(\partial\hat\Omega\cap K_\delta)\geq \mathcal H^2(\partial\hat\Omega_k\cap K)-o(1),\quad k\to\infty.
$$
Taking $k\to\infty$ and then $\delta\to 0$, we conclude
\begin{equation}\label{Eq: upper semicontinuous}
\hat\mu(K)\geq \limsup_{k\to\infty}\hat\mu_k(K),\quad \forall\,\, K\,\,\text{compact}.
\end{equation}
From lower semi-continuity of perimeter, we also have
\begin{equation}\label{Eq: lower semi cotinuous}
\hat\mu(U)\leq \liminf_{k\to\infty}\hat\mu_k(U),\quad \forall\,\,U\,\,\text{open}.
\end{equation}
Now \eqref{Eq: Radon measure convergence} comes from a standard approximation argument using \eqref{Eq: upper semicontinuous} and \eqref{Eq: lower semi cotinuous}.
\end{proof}

The following proposition gives a description for boundary of locally minimizing $\mu$-bubble $\hat\Omega$.
\begin{proposition}\label{Prop: torus component}
The boundary $\partial\hat\Omega$ has a torus component $\hat\Sigma$ contained in a level set of $\phi$. Furthermore, the induced metric of $\hat\Sigma$ is flat and $g$ has constant sectional curvature $1$ at any point of $\hat\Sigma$. We also have $\mathrm d\phi(\nu)=4/\pi$, where $\nu$ is the outward unit normal vector field on $\hat\Sigma$.
\end{proposition}
\begin{proof}
From the definition of $\mathcal F_{\mathring M}$ the boundary of $\hat\Omega_k$ has a component $\hat\Sigma_k$ with nonzero genus. It follows from the second variation formula that
\begin{equation}\label{Eq: second variation hat Sigma k}
0\geq 4\pi\chi(\hat\Sigma_k)\geq \int_{\hat\Sigma_k}4+(h_k\circ\phi)^2+\nu_k(h_k\circ\phi)\,\mathrm d\sigma_g,
\end{equation}
where $\nu_k$ is the outward unit normal vector field of $\hat\Sigma_k$.
Combined with Lemma \ref{Lem: function h epsilon}, surface $\hat\Sigma_k$ must have a non-empty intersection with the compact subset $K=\{-2/3\leq \phi\leq 2/3\}$. Since $\hat\Omega_k$ is a local minimizer of $\mathcal A^{h_k}$ and $h_k$ converges to $h$ smoothly, for any region $U$ with compact closure we have $\area(\hat\Sigma_k\cap U)\leq C$ for some universal constant $C$. From the curvature estimate for stable $h_k$-surfaces (\cite[Theorem 3.6]{ZZ2018}) and elliptic estimates, $\hat\Sigma_k$ converges smoothly to a surface $\hat\Sigma$ with $\hat\Sigma\cap K\neq \emptyset$ in locally graphical sense after passing to a subsequence. It follows from \eqref{Eq: h epsilon}, \eqref{Eq: second variation hat Sigma k} and the fact $\lip\phi=4/\pi$ that
\begin{equation*}
\begin{split}
\int_{\hat\Sigma_k}\left|\mathrm d\phi(\nu_k)-\frac{4}{\pi}\right||h_k'\circ \phi|\,\mathrm d\sigma_g&\leq -\int_{\hat\Sigma_k}4+(h_k\circ\phi)^2+\frac{4}{\pi}h_k'\circ\phi\,\mathrm d\sigma_g\\
&= \int_{\hat\Sigma_k}{\pi\epsilon_k}{\cos^{-2}(2\alpha_{\epsilon_k}\circ\phi)} \left(2\cos\left(\frac{\pi}{2}\phi\right)-1\right)\mathrm d\sigma_g\\
&\leq C\epsilon_k\area(\hat\Sigma_k\cap K)\to 0,\quad k\to \infty,
\end{split}
\end{equation*}
where we drop the negative part in the last inequality.
Denote $\nu$ to be the limit of $\nu_k$, then we have $\mathrm d\phi(\nu)=4/\pi$ on $\hat\Sigma$, which yields $\nabla_{\hat\Sigma}\phi=0$. Therefore surface $\hat\Sigma$ has a closed component $\hat\Sigma'$ contained in $\phi^{-1}(t_0)$ with some $t_0\in[-2/3,2/3]$. From \eqref{Eq: Radon measure convergence} we see $\hat\Sigma\subset \partial\hat\Omega$ and hence $\hat\Sigma'$ is a connected component of $\partial\hat\Omega$. In particular, there is an open neighborhood $V$ of $\hat\Sigma'$ such that $\hat\Sigma\cap V=\hat\Sigma'$. Using the same argument as in \cite[Proposition B.1]{MN2012}, combined with the connectedness of $\hat\Sigma_k$, we conclude that $\hat\Sigma_k$ is a graph over $\hat\Sigma'$ for $k$ sufficiently large. Hence $\hat\Sigma=\hat\Sigma'$ is a surface with nonzero genus. Now we can run a standard analysis on the second variation formula to obtain desired consequences. Since
$$
0\geq 4\pi\chi(\hat\Sigma)\geq\int_{\hat\Sigma}4+(h\circ\phi)^2+\nu(h\circ\phi)\,\mathrm d\sigma_g\geq 0,
$$
all the inequalities are in fact equalities, which yields that $\hat\Sigma$ is a torus, $\mathrm d\phi(\nu)=4/\pi$ and that $g$ has constant sectional curvature $1$ at any point of $\hat\Sigma$. Denote $\hat R$ to be the scalar curvature of $\hat\Sigma$ with induced metric, then the Jacobi operator becomes $-\Delta_{\hat\Sigma}+\hat R$ and hence the first eigenfunction is nonzero constant function combined with the $\mathcal A^h$-stability. It follows that the induced metric of $\hat\Sigma$ is flat.
\end{proof}

From above proposition we can construct a foliation around $\hat\Sigma$.

\begin{lemma}
There is a foliation $\{\tilde\Sigma_s\}_{-\epsilon< s< \epsilon}$ with $\tilde\Sigma_0=\hat\Sigma$ such that
\begin{itemize}
\item each $\tilde\Sigma_s$ is a graph over $\hat\Sigma$ with graph function $\tilde u_s$ along outward unit normal vector field $\nu$ such that
    \begin{equation}\label{Eq: function tilde us}
    \left.\frac{\partial}{\partial s}\right|_{s=0}\tilde u_s=1\quad \text{\rm and}\quad \fint_{\hat\Sigma}\tilde u_s\,\mathrm d\sigma_g=s;
    \end{equation}
\item $\tilde H_s-h\circ\phi$ is a constant function on $\tilde\Sigma_s$, where $\tilde H_s$ is the mean curvature of $\tilde \Sigma_s$.
\end{itemize}
\end{lemma}
\begin{proof}
Let $H_u$ be the mean curvature of the graph over $\hat\Sigma$ with graph function $u$ along outward unit normal vector field $\nu$. We can define a map $\Psi:C^{2,\alpha}(\hat\Sigma)\to \mathring C^\alpha(\hat\Sigma)\times \mathbf R$ such that
\begin{equation*}
 \Psi(u)=\left(H_u-h\circ\phi-\fint_{\hat\Sigma}H_u-h\circ\phi\,\mathrm d\sigma_g,\fint_{\hat\Sigma}u\,\mathrm d\sigma_g\right),
\end{equation*}
where
$$
\mathring C^\alpha(\hat\Sigma)=\{f \in C^\alpha(\hat\Sigma):\int_{ \hat\Sigma}f\,\mathrm d\mu=0\}.
$$
The linearized operator of $\Psi$ at $u=0$ is
\begin{equation*}
D\Psi|_{u=0}(v)=\left(-\Delta_{\hat\Sigma} v,\fint_{\hat\Sigma}v\,\mathrm d\sigma_g\right).
\end{equation*}
Clearly this operator is invertible and it follows from inverse function theorem that there is a family of functions $\tilde u_s$ with $-\epsilon< s< \epsilon$ such that $H_{\tilde u_s}-h\circ\phi$ is a constant function and
$$
\left.\frac{\partial}{\partial s}\right|_{s=0}\tilde u_s=1,\quad \fint_{\hat\Sigma}\tilde u_s\,\mathrm d\sigma_g=s.
$$
With a smaller $\epsilon$, graphs $\tilde\Sigma_s$ over $\hat\Sigma$ with graph function $\tilde u_s$ give the desired foliation.
\end{proof}

The foliation provides us a family of local minimizers for functional $\mathcal A^h$.

\begin{proposition}\label{Prop: minimizing foliations}
Assume that $\partial\hat\Omega-\hat\Sigma$ is disjoint to the region enclosed by $\hat\Sigma$ and $\tilde\Sigma_s$. Let $\tilde\Omega_s$ be the region enclosed by $\partial\hat\Omega-\hat\Sigma$ and $\tilde\Sigma_s$. Then $\tilde\Omega_s$ is still a local minimizer for $\mathcal A^h$.
\end{proposition}
\begin{proof}
We only give a proof for $\tilde\Sigma_s$ with $s>0$ since the argument is valid when $s<0$ as well. First let us show that $\tilde H_\tau-h\circ\phi$ is identically zero for any $0<\tau<s$. Otherwise there is a $\tau$ such that $\tilde H_\tau-h\circ\phi$ is positive since $\hat\Omega$ is a local minimizer for $\mathcal A^h$. Assume that $\hat\Sigma$ is contained in $ \phi^{-1}(t_0)$. We consider the following differential equation
\begin{equation*}
\frac{4}{\pi}\bar h_\delta'+\bar h_\delta^2=-4+\delta,\quad \bar h_\delta(t_0)=h(t_0)+\delta,\quad \delta>0.
\end{equation*}
With $\delta$ chosen to be small enough, we have $\tilde H_0<\bar h_\delta\circ\phi$ on $\hat\Sigma$ and $\tilde H_\tau>\bar h_\delta\circ\phi$ on $\hat\Sigma_\tau$. Therefore, we can find a smooth minimizer $\bar\Omega$ for $\mathcal A^{\bar h_\delta}$ in the region enclosed by $\hat\Sigma$ and $\tilde\Sigma_\tau$ with $\hat\Sigma\subset \partial\bar\Omega$. Since the projection map restricted to $\partial\bar\Omega-\hat\Sigma$ onto $\hat\Sigma$ has degree one, $\partial\bar\Omega-\hat\Sigma$ has a component $\bar\Sigma$ with nonzero genus. It follows from the second variation formula that
\begin{equation*}
0\geq 4\pi\chi(\bar\Sigma)\geq\int_{\bar\Sigma}4+(h\circ\phi)^2+\nu(h\circ\phi)\,\mathrm d\sigma_g\geq \delta\area(\bar\Sigma)>0,
\end{equation*}
which yields a contradiction. Now it is quick to see that $\tilde\Omega_s$ is still a local minimizer for $\mathcal A^h$. Denote $\Omega_\tau$ to be the region enclosed by $\hat\Sigma$ and $\tilde\Sigma_\tau$. Notice that for any region $U$ with compact closure containing $\Omega_s$ there holds
\begin{equation*}
\begin{split}
\mathcal A_U^h(\tilde\Omega_s)&=\mathcal A_U^h(\hat\Omega)+\area(\tilde\Sigma_s)-\area(\hat\Sigma)-\int_{\Omega_s}h\circ\phi\,\mathrm d\mu_g\\
&=\mathcal A_U^h(\hat\Omega)+\int_0^s\mathrm d\tau\int_{\tilde\Sigma_\tau}(\tilde H_\tau-h\circ\phi)\tilde f_\tau\,\mathrm d\sigma_g\\
&=\mathcal A_U^h(\hat\Omega),
\end{split}
\end{equation*}
where $\tilde f_\tau$ is the lapse function of $\tilde\Sigma_\tau$ moving along the foliation. If $\Omega$ satisfies $\Omega\Delta\tilde\Omega_s\Subset U$, then
$$
\mathcal A_U^h(\Omega)\geq \mathcal A_U^h(\hat\Omega)=\mathcal A^h_U(\tilde\Omega_s).
$$
Therefore we complete the proof.
\end{proof}
Now we are ready to prove the rigidity.
\begin{proof}
First we show that $\tilde\Sigma_s$ are $s$-equadistant surfaces to $\hat\Sigma$ contained in level sets of $\phi$. With the foliation we can write the metric as $g=\tilde f_s^2\mathrm ds^2+\tilde g_s$. Since $\tilde \Omega_s$ is $\mathcal A^h$-minimizing, $\tilde\Sigma_s$ satisfies the properties in Proposition \ref{Prop: torus component}. It follows from $\mathcal A^h(\tilde\Omega_s)\equiv\mathcal A^h(\hat\Omega)$ and the second variation formula that $-\Delta\tilde f_s=0$, which implies that $\tilde f_s$ is a constant function. Combined with \eqref{Eq: function tilde us}, $\tilde f_s\equiv 1$ and $\tilde\Sigma_s$ is $s$-equadistant surface to $\hat\Sigma$. Since $\mathrm d\phi(\nu)=4/\pi$ and $\hat\Sigma$ is contained in level sets of $\phi$, so is $\tilde\Sigma_s$.
In particular, we know that $\{\tilde\Sigma_s\}_{-\epsilon<s<\epsilon}$ is a geodesic flow starting from $\hat\Sigma$. Extending it to a maximum solution $\{\tilde\Sigma_s\}_{a<s<b}$, we are going to show
\begin{itemize}
\item $\tilde\Sigma_s$ does not intersect $\partial\hat\Omega-\hat\Sigma$;
\item $\tilde\Sigma_s$ satisfies properties in Proposition \ref{Prop: torus component};
\item Denote $\phi(s)$ to be the value of $\phi$ on $\tilde\Sigma_s$, then $a$ and $b$ satisfy
\begin{equation}\label{Eq: limit phi}
\lim_{s\to a^+}\phi(s)=-1\quad \text{and}\quad \lim_{s\to b^-}\phi(s)=1.
\end{equation}
\end{itemize}
We verify these one by one. Assume by contradiction that one slice $\tilde\Sigma_{s_0}$ has non-empty intersection with $\partial\hat\Omega-\hat\Sigma$, we can take $s_0$ to be the first such moment. Namely, for any $s$ between $0$ and $s_0$ surface $\tilde\Sigma_s$ has no intersection with $\partial\hat\Omega-\hat\Sigma$. For convenience we assume $s_0>0$. Denote $\tilde\Omega_s$ to be the region enclosed by $\tilde\Sigma_s$ and $\partial\hat\Omega-\hat\Sigma$. We define
\begin{equation*}
\mathcal S=\{0<s<s_0:\,\text{$\tilde\Omega_s$ is locally $\mathcal A^h$-minimizing}\}.
\end{equation*}
From previous discussion, $\mathcal S$ is non-empty and relatively open in $(0,s_0)$. Notice that we have
$$
\sup_{0\leq s\leq s_0}\sup_{\tilde\Sigma_s}|\phi|\leq t_0<1
$$
from the smoothness of $\phi$. Combined with Proposition \ref{Prop: torus component}, the principal curvatures of $\tilde\Sigma_s$ with $s\in\mathcal S$ satisfies
$$
\lambda_1\lambda_2=-1\quad \text{and}\quad |\lambda_1+\lambda_2|\leq C(t_0),
$$
which implies $|A|\leq C(t_0)$. Furthermore, we have uniform area bound for $\tilde \Sigma_s$ with $s\in\mathcal S$ due to the locally $\mathcal A^h$-minimizing property. Therefore $\{\tilde\Sigma_s:s\in S\}$ is compact in $C^\infty$ topology and $\mathcal S$ is relatively closed in $(0,s_0)$. This yields $\mathcal S=(0,s_0)$. Applying the compactness again, we know that $\tilde\Omega_{s_0}$ is $\mathcal A^h$-minimizing and then smooth, which leads to a contradiction since $\tilde \Sigma_{s_0}$ intersects other boundary components of $\tilde\Omega_{s_0}$. If $\tilde\Sigma_s$ has no intersection with $\partial\hat\Omega-\hat\Sigma$, then above argument actually tells us that all $\tilde\Omega_s$ are locally $\mathcal A^h$-minimizing and $\tilde\Sigma_s$ satisfies the properties in Proposition \ref{Prop: torus component}. In particular, $\phi(s)$ is monotone increasing and the limits in \eqref{Eq: limit phi} make sense. If \eqref{Eq: limit phi} does not hold, then we can still obtain the compactness from the properness of $\phi$. Then the geodesic flow can be extended through endpoints, which is impossible.

From the connectedness of $\mathring M$ it follows $\mathring M=(a,b)\times\hat\Sigma$. Let $\rho=\pi\phi/4$, then it is a smooth distance function with range $(-\pi/4,\pi/4)$ and $h\circ\phi=-2\tan(2\rho)$. Therefore we can write
\begin{equation*}
\mathring M= \left(-\frac{\pi}{4},\frac{\pi}{4}\right)\times\hat\Sigma\quad\text{and}\quad  g=\mathrm d\rho^2+\tilde g_\rho,
\end{equation*}
where $\tilde g_\rho$ are flat metrics on $\hat\Sigma$ and $g$ has constant sectional curvature 1. Denote $\hat\Sigma_\rho=\hat\Sigma\times\{\rho\}$, then $\hat\Sigma_\rho$ has constant mean curvature $-2\tan(2\rho)$. Combined with the Gauss equation, we conclude that the principal curvatures of $\hat\Sigma_\rho$ are
\begin{equation}\label{Eq: principal curvatures}
\lambda_1=\cot\left(\rho+\frac{\pi}{4}\right) \quad\text{and}\quad \lambda_2=-\tan\left(\rho+\frac{\pi}{4}\right).
\end{equation}
Now we lift the metric to the universal covering $\mathring M_0=(-\pi/4,\pi/4)\times\mathbf R^2$, denoted by $g_0=\mathrm d\rho^2+\bar g_\rho$. Fixing a coordinate $(x,y)$ on $\mathbf R^2$ such that $\bar g_0=\mathrm dx^2+\mathrm dy^2$. From \eqref{Eq: principal curvatures} the second fundamental form of $\{\rho=0\}$ is
\begin{equation}
A=-\cos\left(2\theta(x,y)\right)\mathrm dx^2+2\sin\left(2\theta(x,y)\right)\mathrm dx\mathrm dy+\cos\left(2\theta(x,y)\right)\mathrm dy^2.
\end{equation}
Since $g_0$ has constant curvature $1$, the Codazzi equation gives
$$
\sin\left(2\theta(x,y)\right)\theta'_{y}(x,y)=\cos\left(2\theta(x,y)\right)\theta'_x (x,y)
$$
and
$$
-\sin\left(2\theta(x,y)\right)\theta'_x(x,y)=\cos\left(2\theta(x,y)\right)\theta'_ y(x,y).
$$
It follows that $\theta(x,y)$ is a constant function and we can assume $A=\mathrm dx^2-\mathrm dy^2$ after a possible rotation. From the Jacobi equation we conclude
$$
g_0=\mathrm d\rho^2+2\sin^2\left(\rho+\frac{\pi}{4}\right)\mathrm dx^2+2\cos^2\left(\rho+\frac{\pi}{4}\right)\mathrm dy^2.
$$
This yields $\mathring M=\mathring M_0/\Gamma$ for some lattice $\Gamma$ and we complete the proof.
\end{proof}

\section{Proof of corollaries}\label{Sec: proof corollaries}
In this section, we give a proof for Corollary \ref{Cor: width estimate} and Corollary \ref{Cor: focal estimate}. First let us show the following lemma:

\begin{lemma}\label{Lem: contracting map}
Assume $(M,g)$ is a smooth Riemannian manifold such that $\width(M,g)>2l$, then there exists a surjective smooth map $\phi:(M,g)\to [-l,l]$ with $\lip\phi<1$ such that $\phi^{-1}(-l)=\partial_- M$ and $\phi^{-1}(l)=\partial_+ M$.
\end{lemma}
\begin{proof}
Let $\rho_-(x)=\dist(x,\partial_-M)$. The strict inequality for width implies that there is a small positive constant $\epsilon$ such that $\width(M,g)>2l+4\epsilon$. Fix such an $\epsilon$ and define
\begin{equation*}
\phi_1=\min\left\{\max\{\rho_--l-2\epsilon,-l-\epsilon\},l+\epsilon\right\}.
\end{equation*}
Clearly $\phi_1$ has the following properties:
\begin{itemize}
\item the Lipschitz constant $\lip\phi_1\leq 1$;
\item $\phi_1\equiv -l-\epsilon$ around $\partial_-M$ and $\phi_1\equiv l+\epsilon$ around $\partial_+M$.
\end{itemize}
Through a standard mollification procedure (for example refer to \cite[Appendix]{MV2013}), for any positive constant $\delta$ there is a smooth function $\phi_2$ satisfying $|\mathrm d\phi_2|\leq 1+\delta$ and the second property above. Let
$$
\phi_3=\frac{l}{l+\epsilon}\phi_2,
$$
then we have
$$
|\mathrm d\phi_3|=\frac{l}{l+\epsilon}|\mathrm d\phi_2|\leq \frac{l}{l+\epsilon}(1+\delta).
$$
With $\delta$ prescribed small enough, we have $|\mathrm d\phi_3|\leq c<1$ for some constant $c$. The only thing that we need to guarantee now is $\phi_3^{-1}(- l)=\partial_- M$ and $\phi_3^{-1}(l)=\partial_+M$. This can be done by passing to a perturbation of $\phi_3$. In detail, we take a smooth function $\eta$ vanishing at the boundary but positive at every interior point of $M$ and define $\phi=(1-\epsilon'\eta)\phi_3$ with $\epsilon'$ a positive constant. If $\epsilon'$ is small enough, we can guarantee $|\mathrm d\phi|\leq (c+1)/2<1$ and moreover $\phi$ satisfies our requirements. To see $\phi^{-1}(l)=\partial_+M$, it suffices to show $\phi^{-1}(l)\subset\partial_+M$ since the other direction is clear. For any point $x$ with $\phi(x)=l$, the only possibility is $\phi_3(x)=l$ and $\eta(x)=0$, which yields $x\in \partial_+M$. A similar argument gives $\phi^{-1}(-l)\subset\partial_-M$. It follows from the construction that $\phi:M\to[-l,l]$ is surjective.
\end{proof}

We now give the proof for Corollary \ref{Cor: width estimate}.
\begin{proof}[Proof for Corollary \ref{Cor: width estimate}]
We only need to prove the width estimate when $\inf \Sec(g)>0$. Without loss of generality, we assume $\inf \Sec(g)=1$. If the width estimate does not hold, we obtain from Lemma \ref{Lem: contracting map} a surjective smooth function $\phi: M\to [-\pi/4,\pi/4]$ with $\lip\phi<1$ such that $\phi^{-1}(-\pi/4)=\partial_-M$ and $\phi^{-1}(\pi/4)=\partial_+M$. Denote $\mathring M$ to be the interior of $M$ and $\tilde\phi=4\phi/\pi$, then $\tilde\phi:\mathring M\to (-1,1)$ is surjective and proper. Recall that $M$ is an overtorical band, $\tilde\phi^*([\pt])$ must be non-spherical. Therefore $\tilde\phi$ is an element in $\mathcal F_{\mathring M}$ with $\lip\tilde\phi<4/\pi$, which leads to a contradiction to \eqref{Eq: main}.
\end{proof}

The following is the proof for Corollary \ref{Cor: focal estimate}.

\begin{proof}[Proof for Corollary \ref{Cor: focal estimate}]
Pull back the metric on $\nu_{\Sigma,r_f}$, we obtain an open manifold $(\mathring M,\tilde g)$ with $\mathring M=T^2\times (-r_f,r_f)$ and $\Sec(\tilde g)\geq 1$. Denote $\rho$ to be the signed distance function induced by the second component and $\phi=r_f^{-1}\rho$. It follows from \eqref{Eq: main} that $r_f^{-1}\geq 4/\pi$, which is the desired focal radius estimate. If the equality holds, then $(\mathring M,g)$ is the quotient $\mathring M_0/\Gamma$ with $\Gamma$ a lattice of $\mathbf R^2$. In particular, the focal radius is attained along any unit normal vector of $\Sigma$, which implies that the closure of the image $\exp^\bot(\nu_{\Sigma,r_f})$ must be the entire $M$. From the smoothness we know that $g$ also has constant curvature $1$. Notice that the immersed surface $\Sigma$ is flat, the conclusion follows from local rigidity result \cite[Corollary 3]{Lawson1969}.
\end{proof}

It is worth mention that a different proof for Corollary \ref{Cor: focal estimate} is pointed out to the author by professor Andr\'e Neves, which can be also used to verify Conjecture \ref{Conj: focal radius high dimension} in the case that either of $p$ and $q$ is $1$. We present it below as the end of this section.

\begin{proof}[An alternative proof for Corollary \ref{Cor: focal estimate}]
Assume otherwise the focal radius $r_f>\pi/4$, we are going to show that the principal curvatures of $\Sigma$ have absolute value less than $1$. Once this is done, the induced metric of $\Sigma$ will have positive curvature, which contradicts to the Gauss-Bonnet formula. As before we pull back the metric onto $\nu_{\Sigma,r_f}$ to obtain an open manifold $(\mathring M,\tilde g)$ with $\mathring M=\Sigma\times (-r_f,r_f)$ and $\Sec(\tilde g)\geq 1$. Denote $\rho$ to be the signed distance function to $\Sigma$. It follows from the Riccati equation that
$$
\partial_{\rho}\Hess\rho+\Hess^2\rho=-R(\cdot,\partial_\rho,\partial_\rho,\cdot).
$$
Let $\gamma:[0,r_f)\to\mathring M$ be any normal geodesic starting from $\Sigma$ with $\gamma'(0)=\nabla\rho$ and $\lambda(s)$ the least eigenvalue of $\Hess\rho$ at $\gamma(s)$. Then $\lambda(s)$ is locally Lipschitz and hence differentiable at almost every point. For all these points the function $\lambda(s)$ satisfies
$
\lambda'(s)+\lambda^2(s)\leq -1
$,
which implies
\begin{equation}\label{Eq: principal curvature upper bound}
\lambda(s)\leq\frac{\lambda(0)-\tan s}{1+\lambda(0)\tan s},\quad \forall\,s\in[0,r_f).
\end{equation}
Therefore the principal curvatures of $\Sigma$ with respect to $\nabla\rho$ is greater than $-1$. With the same argument to the signed distance function $-\rho$, we obtain the desired estimate. If $r_f=\pi/4$, then $\Sigma$ is flat and the principal curvatures at all points of $\Sigma$ are $-1$ and $1$. Furthermore, \eqref{Eq: principal curvature upper bound} takes its equality and $\Sigma_s$ has a principal curvature $-\tan(s+\pi/4)$, where $\Sigma_s=\{\rho=s\}$. Denote $\mu(s)$ to be the largest eigenvalue of $\Hess\rho$. It follows a similar argument that
\begin{equation}\label{Eq: principal curvature upper bound 2}
\mu(s)\leq \tan\left(\frac{\pi}{4}-s\right),\quad s\in[0,\pi/4).
\end{equation}
Applying the Gauss-Bonnet formula to $\Sigma_s$, we obtain the equality for \eqref{Eq: principal curvature upper bound 2} as well. It follows from the Gauss and Ricatti equation that $\tilde g$ has constant curvature $1$ and so is $g$. Now the proof is again completed by \cite[Corollary 3]{Lawson1969}.
\end{proof}

\section{Lipschitz constant estimate under positive Ricci curvature}

In this section, we are going to show the following improvement for our main theorem:
\begin{theorem}
Assume $(\mathring M^3,g)$ is a connected orientable open Riemannian manifold with uniformly positive Ricci curvatures and non-empty $\mathcal F_{\mathring M}$. Then we have
\begin{equation}\label{Eq: lip Ric}
\left(\frac{1}{2}\inf \Ric(g)\right)^{\frac{1}{2}}\cdot\lip\phi\geq \frac{4}{\pi}
\end{equation}
for any $\phi\in \mathcal F_{\mathring M}$. Up to rescaling the equality holds if and only if $(\mathring M,g)$ is isometric to $\mathring M_0/\Gamma$ for a lattice $\Gamma$ of $\mathbf R^2$ and $\phi$ is a multiple of the standard signed distance function on $\mathring M_0/\Gamma$.
\end{theorem}
\begin{proof}
The proof is almost identical to that of Theorem \ref{Thm: main}, so we just emphasize on the difference. As before we can assume $\Ric(g)\geq 2g$ from rescaling. To show \eqref{Eq: lip Ric} we repeat the steps in Section \ref{Sec: proof 1.4} and adopt the same notation here. Without any change $\hat\Sigma$ has positive genus and the second variation formula yields
\begin{equation*}
\int_{\hat\Sigma}|\nabla\psi|^2-(\Ric(\nu,\nu)+|A|^2+\nu(h \circ \phi))\psi^2\,\mathrm d\sigma_g\geq 0,\quad \forall\,\psi\in C^\infty(\hat\Sigma).
\end{equation*}
The key observation here is that we can organize the left side into a form where only the Ricci curvature is involved. Just like in \cite[Section 2]{Ros06}, we can write
\begin{equation}\label{Eq: second variation new}
\Ric(\nu,\nu)+|A|^2=\Ric(e_1,e_1)+\Ric(e_2,e_2)+(h\circ\phi)^2-\hat R,
\end{equation}
where $\{e_1,e_2\}$ is an arbitrary orthonormal frame on $\hat\Sigma$. By taking the test function $\psi\equiv 1$ and using the fact $\Ric(g)\geq 2g$, we see
\begin{equation*}
0\geq 4\pi\chi(\hat\Sigma)=\int_{\hat\Sigma}\hat R\,\mathrm d\sigma_g\geq \int_{\hat\Sigma}4+(h\circ\phi)^2+\nu(h\circ\phi)\,\mathrm d\sigma_g.
\end{equation*}
The particular choice of $\beta$ as in the proof for \eqref{Eq: main} gives \eqref{Eq: lip Ric}.

In order to prove the rigidity we follow the argument in section \ref{Sec: rigidity}. In the same philosophy, based on \eqref{Eq: second variation new} we can prove a similar result to Proposition \ref{Prop: torus component}:
\begin{proposition}
The boundary $\partial\hat\Omega$ has a torus component $\hat\Sigma$ contained in a level set of $\phi$. Furthermore, the induced metric of $\hat\Sigma$ is flat and $\Ric_g(v,v)=2$ for any unit vector $v\in T\hat\Sigma$. We also have $\mathrm d\phi(\nu)=4/\pi$, where $\nu$ is the outward unit normal vector field on $\hat\Sigma$.
\end{proposition}
Using this proposition we can conduct the standard foliation argument as before to obtain the following:
\begin{itemize}
\item[(1)] the open manifold $\mathring M$ splits as
\begin{equation*}
\mathring M= \left(-\frac{\pi}{4},\frac{\pi}{4}\right)\times\hat\Sigma\quad\text{and}\quad  g=\mathrm d\rho^2+\tilde g_\rho.
\end{equation*}
\item[(2)] $\Ric_g(v,v)=2$ for any unit vector $v$ orthogonal to $\partial_\rho$.
\end{itemize}
It follows from (2) that the curvature $R(v,\partial_\rho,\partial_\rho,v)$ at a fixed point $p$ is independent of the choice of the unit vector $v\in T_p\mathring M$ orthogonal to $\partial_\rho$. Fix such a unit vector $v$ and take another unit vector $w\in T_p\mathring M$ orthogonal both to $v$ and $\partial_\rho$. It follows that
\begin{equation*}
\begin{split}
R(v,\partial_\rho,\partial_\rho,v)&=\frac{1}{2}\left(R(v,\partial_\rho,\partial_\rho,v)+ R(w,\partial_\rho,\partial_\rho,w)\right)\\
&=\frac{1}{2}\Ric(\partial_\rho,\partial_\rho)\geq 1
\end{split}
\end{equation*}
for any unit vector $v$ orthogonal to $\partial_\rho$. Using this we can investigate the equation
$$
\partial_{\rho}\Hess\rho+\Hess^2\rho=-R(\cdot,\partial_\rho,\partial_\rho,\cdot)
$$
as in the alternative proof for Corollary \ref{Cor: focal estimate}. Eventually we can obtain $R(v,\partial_\rho,\partial_\rho,v)\equiv 1$ for any unit vector $v$ orthogonal to $\partial_\rho$ and so the metric $g$ has constant sectional curvature $1$. The rest of the proof is identical to the original one and we complete the proof.
\end{proof}
\bibliography{WidthEstimate}
\bibliographystyle{amsplain}
\end{document}